\documentclass{amsart}
\usepackage[utf8]{inputenc}
\usepackage{amsmath, amssymb, amsfonts, tikz, hyperref, url}
\usepackage{amsthm}
\usepackage{tikz-cd}
\usepackage{cleveref}
\usepackage{comment}
\newtheorem{theorem}{Theorem}[section]
\newtheorem{lemma}[theorem]{Lemma}
\newtheorem{prop}[theorem]{Proposition}

\theoremstyle{definition}
\newtheorem{definition}[theorem]{Definition}

\theoremstyle{remark}
\newtheorem{remark}[theorem]{Remark}

\DeclareMathOperator{\GL}{GL}

\DeclareMathOperator{\pr}{pr}

\DeclareMathOperator{\Spec}{Spec}

\renewcommand{\AA}{\mathbb{A}}

\newcommand{\FF}{\mathcal{F}}

\newcommand{\CC}{\mathbb{C}}

\newcommand{\EE}{\mathcal{E}}

\newcommand{\OO}{\mathcal{O}}

\newcommand{\LL}{\mathcal{L}}

\newcommand{\PP}{\mathbb{P}}

\begin{document}

\title{A Serre Relation in the  $K$-theoretic Hall algebra of surfaces}
\date{}
\author{Junyao Peng, Yu Zhao}
\maketitle

\begin{abstract}
    We prove a Serre relation in the $K$-theoretic Hall algebra of surfaces
    constructed by Kapranov-Vasserot \cite{kapranov2019cohomological} and the second author \cite{zhao2019k}. 
\end{abstract}

\section{Introduction}

Given a smooth quasi-projective surface over $\mathbb{C}$, an associative
algebra structure on
\[
K(Quot) = \bigoplus_{n\geq 1} K_{\GL_n}(Quot_n^\circ),
\]
was constructed by Kapranov-Vasserot
\cite{kapranov2019cohomological} and the second author \cite{zhao2019k} and
called the $K$-theoretic Hall algebra of a surface. It was inspired by
Schiffmann-Vasserot \cite{schiffmann2013} for the case $S=\mathbb{A}^2$ and  constructed by
Sala-Schiffmann \cite{sala2018cohomological} and Alexander Minets
\cite{minets18:cohom_hall_higgs} when $S$ is the cotangent bundle of an algebra curve $C$. The $K$-theoretic Hall algebra was categorified by Porta-Sala
\cite{porta2019categorification}, and the two-dimensional categorified Hall
algebra was studied by Diaconescu-Porta-Sala \cite{diaconescu2020mckay} when $S$ is the crepant resolution
of type $A$ singularities.

When $S=\mathbb{A}^2$ with equivariant $\mathbb{G}_m^2$ action, the
$K$-theoretic Hall algebra could be identified by the positive part of the
elliptic Hall algebra, i.e. $\mathbb{Z}[q_1,q_2]$-algebra with generators  $\{E_{k}\}_{k\in \mathbb{Z}}$  modulo the following relations:
  \begin{multline}
    \label{2.1}
    (z-wq_1)(z-wq_2)(z-\frac{w}{q})E(z)E(w) =(z-\frac{w}{q_1})(z-\frac{w}{q_2})(z-wq)E(w)E(z)
  \end{multline}
   \begin{equation}
    \label{2.3}
     [[E_{k+1},E_{k-1}],E_k]=0 \quad \forall k\in \mathbb{Z}
   \end{equation}
    where
   \begin{equation}
     E(z)=\sum_{k\in \mathbb{Z}}\frac{E_k}{z^k}
   \end{equation}
The purpose of this note is to show that the Serre relations also exist in the
$K$-theoretic Hall algebra of any surface. We prove that

\begin{theorem}
  \label{main}
  Given an integer $k$, let $e_k=[z^k\mathcal{O}_S]\in K_{GL_1}(Quot_1^{\circ})$ where $z$ is the
  standard representation of $GL_1$. Then we have
  $$[[e_{k+1},e_{k-1}],e_k]=0$$
\end{theorem}

The main idea is inherited from Negut \cite{neguct2018hecke},
while we study the stack case instead of the moduli space of stable sheaves.
Thus we could generalize the Serre relations to other settings, like the
PT categories of local surfaces \cite{Toda2020hall}. We could also remove the Assumption
A and Assumption S in \cite{neguct2018hecke}.

The paper is part of the MIT Undergraduate Research Opportunities Program (UROP), and we are very grateful for Steven Johnson, Elchanan
Mossel and Slava Gerovitch for organizing this program. We also thank Andrei
Negut for suggesting this problem.
\section{Moduli spaces of sheaves and their geometry}
\subsection{Quot Schemes and Flag Schemes}
We follow the notation of \cite{zhao2019k} for Quot and flag schemes.  Given a
non-negative integer $d$, we denote the Grothendieck's Quot scheme 
  $Quot_{d}^{\circ}$ as  the moduli scheme of quotients of coherent sheaves
$$\{\phi_{d}:\mathcal{O}_{S}^{d}\twoheadrightarrow
\mathcal{E}_{d}\},$$
such that
\begin{enumerate}
\item The sheaf  $\mathcal{E}_{d}$ has dimension $0$ and length $d$;
\item  The morphism $H^{0}(\phi):k^{d}\to
H^{0}(\mathcal{E}_{d})$ is an isomorphism.
\end{enumerate}
Over $Quot_{d}^{\circ}\times S$, there is a universal quotient of
coherent sheaves with kernel denoted by $\mathcal{I}_{d}$
$$\phi_{d}:\mathcal{O}^{d} \twoheadrightarrow
\mathcal{E}_{d}.$$
Given be a sequence of non-decreasing integers $d_{\bullet}=(d_{0},d_{1},\ldots,d_{l})$, such
that $d_{0}=0$ and $d_{l}=d$, we fix a flag of inclusion maps $F=\{\mathcal{O}^{d_{1}}\subset \ldots \subset \mathcal{O}^{d_{l}}\}$ and denote $Flag_{d_{\bullet}}^{\circ}$ the moduli space of
  coherent sheaves  $\{\mathcal{E}_{d_{1}}\subset \cdots \mathcal{E}_{d_{l}}\}$
  with quotient maps
  $$\phi_{i}:\mathcal{O}^{d_{i}}\to\mathcal{E}_{d_{i}}$$
  such that
  \begin{enumerate}
  \item $\mathcal{E}_{d_{i}}$ has dimensional $0$ and length $d_{i}$;
  \item $\phi_{i}$ are commutative with the inclusion maps;
  \item $H^{0}(\phi_{i}):k^{d_{i}}\to H^{0}(\mathcal{E}_{d_{i}},k)$.
    are isomorphisms.
  \end{enumerate}
  We also denote 
$$ Quot_{d_{\bullet}}^{\circ }=\prod_{i=1}^{k}Quot_{d_{i}-d_{i-1}}^{\circ }.$$

For each $i$,  over $Flag_{d_{\bullet}}^{\circ}\times S$ there are universal quotients of coherent sheaves
$$\phi_{i}:\mathcal{O}^{d_{i}}\to \mathcal{E}_{d_{i}}.$$
 Fixing an isomorphism
$\mathcal{O}^{d_{i}-d_{i-1}}=\mathcal{O}^{d_{i}}/\mathcal{O}^{d_{i-1}}$ and
defining
$\mathcal{E}_{d_{i},d_{i-1}}:=\mathcal{E}_{d_{i}}/\mathcal{E}_{d_{i-1}}$, we
have 
$$\phi_{i,i-1}:\mathcal{O}^{d_{i}-d_{i-1}}\to
\mathcal{E}_{d_{i},d_{i-1}}$$ is also surjective. It induces a morphism
\begin{equation}
  \label{eq:def2}
  p_{d_{\bullet}}:Flag_{d_{\bullet}}^{\circ }\to Quot_{d_{\bullet}}^{\circ }.
\end{equation}

We will also consider the group actions on Quot schemes and
flag schemes, with the following notations:
\begin{enumerate}
\item The group $GL_{d}=GL_{d}(k)$ has a natural action on $Quot_{d}^{\circ}$ by acting on $\mathcal{O}^{n}$.
\item Let $P_{d_{\bullet}}$ be the parabolic group of $GL_{d}$ which
  preserves the flag $F$. $P_{d_{\bullet}}$ has a natural action on
  $Flag_{d_{\bullet}}^{\circ }$. Let $B_{d_{\bullet}}$ be the parabolic Lie
  subalgebra of $\mathfrak{gl}_{n}$ which preserves the flag $F$.
\item By \cite{minets18:cohom_hall_higgs} $Flag_{d_{\bullet}}^{\circ }$, is
a closed subscheme of $Quot_{d}^{\circ}$ with an inclusion map 
\begin{equation}
  \label{eq:def1}
  i_{d_{\bullet}}:Flag_{d_{\bullet}}^{\circ} \to Quot_{d}^{\circ }.
\end{equation} The morphism $i_{d_{\bullet}}$ is $P_{d_{\bullet}}$-equivariant. 
  Let $\widetilde{Flag_{d_{\bullet}}^{\circ }}=Flag_{d_{\bullet}}^{\circ
  }\times_{P_{d_{\bullet}}}G_{d}$. $i_{d_{\bullet}}$ induces a proper
  $G_{d}$-equivariant
  morphism $$q_{d_{\bullet}}:\widetilde{Flag_{d_{\bullet}}^{\circ
    }}\to Quot_{d}^{\circ}.$$
\item We will use the notation $Quot_{n,m}^{\circ}$ for $Quot_{d_\bullet}^\circ$ where $d_{\bullet}=(0,n,n+m)$ and $Quot_{n,m,l}^\circ=Quot_{d_{\bullet}}^\circ$ for
  $d_{\bullet}=(0,n,n+m,n+m+l)$. The same principle holds for other
  notations, like $Flag_{d_\bullet}^\circ,p_{d_\bullet},q_{d_{\bullet}}$ and so
  on.
\item Given a matrix $X$ (or other notations like $Y,g$, etc.) we will always
  denote $X_{ij}$ (or $Y_{ij},g_{ij}$, etc.) the $i$-th row and $j$-th column of the matrix.
\end{enumerate}

\begin{lemma}
  $\widetilde{Flag_{d_{\bullet}}^{\circ}}$ is the moduli space of coherent sheaves
  $$\{\mathcal{E}_{d_{1}}\subset \cdots \mathcal{E}_{d_{l}}\}$$
  with a quotient map
  $$\phi_d:\mathcal{O}^{d}\to \mathcal{E}_{d}$$
  such that $H^{0}(\phi_d):k^d\to H^0(\mathcal{E},k)$ is an isomorphism.
\end{lemma}
\begin{proof}
  It is obvious as $[\widetilde{Flag_{d_{\bullet}}^{\circ}}/GL_d]=[Flag_{d_{\bullet}}^\circ/P_{d_\bullet}]$ which is the moduli stack of coherent sheaves
  $$\{\mathcal{E}_{d_{1}}\subset \cdots \mathcal{E}_{d_{l}}\}$$
  where $\mathcal{E}_{d_i}$ is a dimension $0$, length $d_i$ coherent sheaf.
\end{proof}

We follow the notation of \cite{neguct2018hecke} for the set partition, i.e. an equivalence relation on a finite ordered set. We represent partitions
 suggestively, for example $(x,y,x)$ will refer to the partition of a
 $3$-element set into distinct $1$-element subsets, while $(x,y,x)$(respectively
 $(x,x,x)$) refers to the equivalence relation which sets the first and the last
 element (respectively all elements) equivalent to each other. The size of a
 partition $\lambda$, which is denoted by $|\lambda|$, is the number of elements
 of the underlying set.

 \begin{definition}
   Given a positive integer $n$, let $d_{\bullet}=(0,1,\cdots,n)$. For a set
   partition $\lambda$ of size $n$, we will consider the schemes
   $Flag_{\lambda}^{\circ}$ (or $\widetilde{Flag_{\lambda}^{\circ}}$) which
   consist of elements in $Flag_{d_{\bullet}}^\circ$ (or $\widetilde{Flag_{d_{\bullet}}^{\circ}}$)
   such that
   $$\{\mathcal{E}_{0}=0\subset_{x_1}\mathcal{E}_{1}\subset_{x_2}\cdots
   \subset_{x_n} \mathcal{E}_{n}\}$$
   for some $x_1,\cdots, x_n\in S$ such that $x_i=x_j$ if $i\sim j\in \lambda$,
   where $\mathcal{F}'\subset_{x}\mathcal{F}$ means that $\mathcal{F}'\subset
   \mathcal{F}$ and $\mathcal{F}/\mathcal{F}'\cong k_x$.
 \end{definition}
 \begin{remark}
   
 On $Flag_{1,1}^\circ$ (or $\widetilde{Flag_{1,1}^\circ}$), there are
tautological vector bundles $\mathcal{U}_1\subset \mathcal{U}_2$ such that the
fiber at each closed point is the vector space of global sections of $\mathcal{E}_{1}$ and
$\mathcal{E}_2$ respectively. Moreover, we define line bundles $\LL_2
=\mathcal{U}_2/\mathcal{U}_{1}=z_2\mathcal{O}$ and $\LL_1 =
\mathcal{U}_1=z_1\mathcal{O}$, where $z_1$ (or $z_2$) is a
character of $P_{1,1}$ by mapping the matrix $X\in P_{11}$ to $X_{11}$ (or $X_{22}$).

  On $Flag_{1,1,1}^\circ$  or $\widetilde{Flag_{1,1,1}^\circ}$, there are universal vector bundles $\mathcal{U}_i$ whose fiber at each closed point is
the vector space of global sections of $\mathcal{E}_i$ and we define line bundles $\LL_i =
\mathcal{U}_{i}/ \mathcal{U}_{i-1}$. $\LL_i=z_i\mathcal{O}$ is $P_{1,1,1}$ equivariant on
$Flag_{1,1,1}^\circ$ and $G_3$ equivariant on  $\widetilde{Flag_{1,1,1}^\circ}$,
where $z_i$ are characters of $P_{1,1,1}$ which map the matrix $X$ to $X_{ii}$.
\end{remark}
The following lemmas state the local geometric properties of Flag schemes. While those lemmas hold for all smooth quasi-projective surfaces, we only need
to prove the case when $S=\mathbb{A}^2$, as the problem is local.
\begin{lemma}\label{lemma:Flag_schemes_Cohen_Macaulay}
The schemes $Flag_{x,y}^\circ$, $Flag_{x,x}^{\circ}$,
$Flag_{x,y,z}^{\circ}$,$Flag_{x,x,y}^\circ$, $Flag_{y,x,x}^\circ$ are Gorenstein schemes of $5,4,9,8,8$ dimensions respectively.
\end{lemma}
\begin{proof} When $S=\mathbb{A}^2$,
\begin{itemize}
    \item $Flag_{x,y}^\circ = \{X,Y\in B_{1,1}: [X,Y] = 0\}$ is a subscheme of $\AA^6 = \AA[X_{ij},Y_{ij}]_{1\leq i\leq j\leq 2}$ cut out by the following equation
    \[
    X_{12}(Y_{11}-Y_{22}) = Y_{12}(X_{11}-X_{22}) 
    \]
    and hence has dimension 5.
    \item $Flag_{x,x}^\circ = \{X,Y\in B_{1,1}: [X,Y] = 0\}$ is a subscheme of
      $\AA^6 = \AA[X_{ij},Y_{ij}]_{1\leq i\leq j\leq 2}$ cut out by the
      following equation
      \begin{align*}
        X_{11}=X_{22} \quad  Y_{11}=Y_{22}
      \end{align*}
      and hence has dimension $4$.
    \item $Flag_{x,y,z}^\circ = \{X,Y\in B_{1,1,1}: [X,Y] = 0\}$ is a subscheme of $\AA^{12} = \AA[X_{ij},Y_{ij}]_{1\leq i\leq j\leq 3}$ cut out by the following equations
    \begin{align*}
        & X_{12}(Y_{11}-Y_{22}) = Y_{12}(X_{11}-X_{22}) \\
        & X_{23}(Y_{22}-Y_{33}) = Y_{23}(X_{22}-X_{33}) \\
        & X_{13}(Y_{11}-Y_{33}) - Y_{13}(X_{11} - X_{33}) = X_{12}Y_{23} - X_{23}Y_{12}
    \end{align*}
    and hence has dimension 9.
   \item $Flag_{x,x,y}^\circ = \{X,Y\in B_{1,1,1}: [X,Y] = 0, X_{22}=X_{33},Y_{22}=Y_{33}\}$ is a subscheme of $\AA^{12}$ cut out by equations
     \begin{align*}
       & X_{22}=X_{33},Y_{22}=Y_{33} \\
        & X_{12}(Y_{11}-Y_{22}) = Y_{12}(X_{11}-X_{22}) \\
        & X_{13}(Y_{11}-Y_{22}) - Y_{13}(X_{11} - X_{22}) = X_{12}Y_{23} - X_{23}Y_{12}
    \end{align*}
    Thus, it has dimension 8. The similar computation also holds for $Flag_{y,x,x}^\circ$.
    \end{itemize}

  \end{proof}
  \begin{lemma}
    $Flag_{x,y,x}^{\circ}$ is Cohen-Macaulay of dimension $8$.
  \end{lemma}
  \begin{proof} When $S=\mathbb{A}^2$,
    $Flag_{x,y,x}^\circ = \{X,Y\in B_{1,1,1}: [X,Y] = 0, X_{11}=X_{33},Y_{11}=Y_{33}\}$ is a subscheme of $\AA^{12}$ cut out by equations
    \begin{align*}
      & X_{11}=X_{33} \quad  Y_{11}=Y_{33} \\
  & X_{12}(Y_{11}-Y_{22}) = Y_{12}(X_{11}-X_{22}) \\
        & X_{23}(Y_{22}-Y_{11}) = Y_{23}(X_{22}-X_{11}) \\
        & X_{12}Y_{23} = X_{23}Y_{12}.
    \end{align*}
    Let
    $x_1=X_{12},x_2=X_{11}-X_{22},y_1=Y_{12},y_{2}=Y_{11}-Y_{22},x_3=X_{23},y_3=Y_{23}$,
    then we only need to prove that
    $$k[x_1,x_2,x_3,y_1,y_2,y_3]/(x_1y_2-x_2y_1,x_2y_3-x_3y_2,x_3y_1-x_1y_3)$$
    is Cohen-Macaulay of dimension $4$, which follows from Claim 5.22 of \cite{neguct2018hecke}.
  \end{proof}

  \begin{lemma}\label{lemma:Flag_schemes_normal}
The schemes $Flag_{x,y}^\circ$, $Flag_{x,x}^{\circ}$,
$Flag_{x,y,z}^{\circ}$,$Flag_{x,x,y}^\circ$, $Flag_{x,y,x}^\circ$,
$Flag_{y,x,x}^\circ$ are normal.
\end{lemma}
\begin{proof}   By Lemma \ref{lemma:Flag_schemes_Cohen_Macaulay},  $Flag_{x,y}^{\circ}$ is cut out by the equation
\[X_{21}(Y_{11}-Y_{22}) - Y_{21}(X_{11}-X_{22}).
\]
By taking partial derivatives, we can compute that the singular locus is locally given by
\[
X_{21}=Y_{21}=X_{11}-X_{22} = Y_{11} - Y_{22} = 0,
\]
which is of codimension 3 and thus smooth.

Similarly, for the other cases, it suffices to show that the singular loci of
$Flag_{x,x,y}^\circ$, $Flag_{x,y,x}^\circ$, and $Flag_{y,x,x}^\circ$ have
codimension $\geq 2$. Here we only prove the $Flag_{x,y,x}^\circ$ case. In the proof of Lemma \ref{lemma:Flag_schemes_Cohen_Macaulay}, we computed that $Flag_{x,y,x}^{\circ}$ is isomorphic to \[\AA^4 \times \AA^6/(x_1y_2-x_2y_1,x_2y_3-x_3y_2,x_3y_1-x_1y_3).\]
  The Jacobian matrix of this ideal is
\[
J = \begin{pmatrix}
y_2 & -y_1 & & -x_2 & x_1 & \\
y_3 & & -y_1 & -x_3 &  & x_1 \\
 & y_3 & -y_2 & & -x_3 & x_2
\end{pmatrix}
\]
At general points, the rank of $J$ is 2. The singular locus is the set of points where $\text{rank}(J)\leq 1$. This only happens when $x_i = y_i = 0$ for all $i$. Thus, the singular locus of $Flag_{x,y,x}^\circ$ has codimension 4, as desired.
\end{proof}
 \subsection{Quadruple and Triple moduli space of sheaves}

\begin{definition}
We define the quadruple moduli space $\mathcal{Y}$ which parameterizes the
following commutative diagram
    \begin{equation}\label{diag:3}
    \begin{tikzcd}
     & \EE_1 \arrow[hookrightarrow]{dr}{y}\\
  0 \ar[hookrightarrow]{ur}{x} \ar[hookrightarrow]{dr}{y}   & & \EE_2  \\
    & \EE_1' \arrow[hookrightarrow]{ur}{x}
    \end{tikzcd}
  \end{equation}
 of coherent sheaves where each successive inclusion is colength $1$ and
 supported at the point indicated on the diagram, with a surjective morphism $\phi:\mathcal{O}^{2}\to \mathcal{E}_2$ such that
  $h^{0}(\phi)$ is an isomorphism.

We define $\mathcal{Y}_+,\mathcal{Y}_{-}$ to be the moduli space which
parameterize the following commutative diagrams:

    \begin{equation}\label{diag:4}
    \begin{tikzcd}
     & \EE_1 \arrow[hookrightarrow]{dr}{y}\\
  0 \ar[hookrightarrow]{ur}{x} \ar[hookrightarrow]{dr}{y}   & & \EE_2 \arrow[hookrightarrow]{r}{x} & \EE_3 \\
    & \EE_1' \arrow[hookrightarrow]{ur}{x}
    \end{tikzcd}
  \end{equation}
  \begin{equation}\label{diag:5}
    \begin{tikzcd}
    & & \EE_2 \arrow[hookrightarrow]{dr}{y} \\
     0 \arrow[hookrightarrow]{r}{x} & \EE_1 \arrow[hookrightarrow]{ur}{x}\arrow[hookrightarrow]{dr}{y} & & \EE_3 \\
    & & \EE_2' \arrow[hookrightarrow]{ur}{x}
    \end{tikzcd}
  \end{equation}
respectively, of coherent sheaves where each successive inclusion is colength $1$ and
  supported at the point indicated on the diagram, with a surjective morphism $\phi:\mathcal{O}^{3}\to \mathcal{E}_3$ such that  $h^{0}(\phi)$ is an isomorphism.
\end{definition}
For the above moduli spaces $\mathcal{Y},\mathcal{Y}_+,\mathcal{Y}_-$, we
denote $\EE_i=\EE_{i}'$ if $\EE_{i}'$ is not in diagrams
(\ref{diag:3})-(\ref{diag:5}). We denote $\mathcal{U}_{i}$ (or $\mathcal{U}_i'$) to be the
locally free sheaves with fibers the vector space of global sections of $\mathcal{E}_i$ (or
$\mathcal{E}_{i}'$) and denote
\begin{align*}
\mathcal{L}_i=\mathcal{U}_{i}/\mathcal{U}_{i-1}; \\
\mathcal{L}_i'=\mathcal{U}_i'/\mathcal{U}_{i-1}'.
\end{align*}

\begin{lemma}\label{lemma:Y,Y+-_reduced}
$\mathcal{Y}, \mathcal{Y}_{\pm}$ are reduced.
\end{lemma}
\begin{proof} We only prove the case $S=\mathbb{A}^2$, as it is still a local question.
Let $\mathcal{Y}'$ be the fiber product
\begin{equation}\label{fig:fiber-product-Y'}
\begin{tikzcd}
\mathcal{Y}' \arrow{r} \arrow{d} & {Flag_{1,1}^\circ}\times \GL_2 \arrow{d}{\phi}\\
{Flag_{1,1}^\circ}\times \GL_2 \arrow{r}{\phi'} & Quot_2^\circ\times S\times S
\end{tikzcd}
\end{equation}
Then $\mathcal{Y}=\mathcal{Y}'/(P_{1,1}\times P_{1,1})$ and to prove
$\mathcal{Y}$ is reduced, we only need prove that
$\mathcal{Y}'$  is reduced.

By (\ref{fig:fiber-product-Y'}), $\mathcal{Y}'$ contains
\[
\{(X,Y,g;X',Y',g'): X,Y,X',Y'\in B_{1,1}, g,g'\in \GL_2\}
\]
such that
\[
\begin{cases}
(X_{11},Y_{11}) = (X_{22}',Y_{22}') , (X_{22},Y_{22}) = (X_{11}',Y_{11}'), \\
gXg^{-1} = g'X'g'^{-1}, gYg^{-1} = g'Y'g'^{-1},\\
XY = YX, X'Y'=Y'X'
\end{cases}
\]
We replace $h=gg'^{-1}$, then  $\mathcal{Y}'$ contains elements in
\[
\{(X,Y,h;X',Y'): X,Y,X',Y'\in B_{1,1}, g\in \GL_2\}\times GL_2
\]
such that
\[
\begin{cases}
(X_{11},Y_{11}) = (X_{22}',Y_{22}'), (X_{22},Y_{22}) = (X_{11}',Y_{11}') \\
hXh^{-1} = X', hYh^{-1} = Y',\\
XY = YX, X'Y'=Y'X'.
\end{cases}
\]
The condition $X'h = hX$ is equivalent to
\[
\begin{cases}
h_{11}(X_{22}-X_{11}) = h_{12}X_{21},\\
h_{11}X_{21}' = h_{22}X_{21},\\
h_{22}(X_{22}-X_{11}) = h_{12}X_{21}'
\end{cases}
\]
and similar for $Y'h=hY$.
The condition $XY = YX$ is equivalent to
\[
X_{21}(Y_{22}-Y_{11}) = Y_{21}(X_{22} -X_{11}).
\]
We have two cases (since $h\in \GL_2$):
\begin{itemize}
    \item $h_{11}\neq 0$ (i.e., in the open subset $\{h_{11}\neq 0\}\cap \GL_2$). We obtain 
    \[
    \begin{cases}
    X_{22}-X_{11} = \frac{h_{12}}{h_{11}}X_{21}, \\
    X_{21}' = \frac{h_{22}}{h_{11}} X_{21},
    \end{cases} \text{ and }\begin{cases}
    Y_{22}-Y_{11} = \frac{h_{12}}{h_{11}}Y_{21}, \\
    Y_{21}' = \frac{h_{22}}{h_{11}} Y_{21}.
    \end{cases} 
    \]
    These equations cut out an affine space.
    \item $h_{12}\neq 0$. We obtain
    \[
    \begin{cases}
    X_{21}' = \frac{h_{22}}{h_{12}}(X_{22}-X_{11}), \\
    X_{21} =  \frac{h_{11}}{h_{12}}(X_{22}-X_{11}),
    \end{cases} \text{ and } \begin{cases}
    Y_{21}' = \frac{h_{22}}{h_{12}}(Y_{22}-Y_{11}), \\
    Y_{21} =  \frac{h_{11}}{h_{12}}(Y_{22}-Y_{11}).
    \end{cases}
    \]
    These equations cut out an affine space.
  \end{itemize}
  Hence $\mathcal{Y}'$ is smooth and thus reduced.
  
For the scheme $\mathcal{Y}_+$, by applying the similar argument, we define
$\mathcal{Y}'_+$ through the Cartesian diagram:
\[
\begin{tikzcd}
{\mathcal{Y}}_+' \arrow{r} \arrow{d} &  Flag_{x,y,x}^{\circ}\times P_{21} \arrow{d}\\
{Flag_{y,x,x}^{\circ}} \arrow{r}{} & {Flag_{2,1}^\circ} \times S\times S
\end{tikzcd}
\]
and we only need to prove that $\mathcal{Y}_+'$ is reduced. We have
$\mathcal{Y}'_+$ contains the elements in
$$\{(g,X,Y;X',Y')\in P_{21}\times B_{111}^4\}$$
which satisfy:
\[
\begin{cases}
X_{11} =X_{33} = X_{11}' = X_{22}', X_{22} = X_{33}', \\
Y_{11} =Y_{33} = Y_{11}' = Y_{22}', Y_{22} = Y_{33}', \\
X' = gXg^{-1}, Y' = gYg^{-1}, \\
XY-YX=X'Y'-Y'X' = 0.
\end{cases}
\]
The condition $X'g = gX$ is equivalent to 
\begin{align*}
&\begin{pmatrix}
g_{11}X_{11} &  & \\
g_{11}X_{21}' + g_{21}X_{11} & g_{22}X_{11} & g_{23}X_{11} \\
g_{11}X_{31}' + g_{21}X_{32}' + g_{31}X_{22} & g_{22}X_{32}' + g_{32}X_{22} & g_{23}X_{32}' + g_{33}X_{22}
\end{pmatrix}\\
=& \begin{pmatrix}
g_{11}X_{11} &  & \\
g_{21}X_{11} + g_{22}X_{21}+ g_{23}X_{31} & g_{22}X_{22}+g_{23}X_{32} & g_{23}X_{11}\\
g_{31}X_{11} + g_{32}X_{21}+ g_{33}X_{31} & g_{32}X_{22}+g_{33}X_{32} & g_{33}X_{11}
\end{pmatrix}
\end{align*}
Since $g_{11}\neq 0$, we can solve $X_{21}',X_{31}'$:
\[
X_{21}' = \frac{1}{g_{11}}(g_{22}X_{21}+g_{23}X_{31})
\]
\[
X_{31}' = \frac{1}{g_{11}}(g_{31}X_{11}+ g_{32}X_{21}+g_{33}X_{31}-g_{21}X_{32}'-g_{31}X_{22})
\]
The other equations are:
\[
\begin{cases}
g_{22}X_{32}' = g_{33}X_{32}, \\
g_{22}(X_{11}-X_{22}) = g_{23}X_{32}, \\
g_{33}(X_{11}-X_{22}) = g_{23}X_{32}'.
\end{cases}
\]
The condition $XY-YX = 0$ gives
\[
\begin{cases}
X_{21}(Y_{11}-Y_{22}) = Y_{21}(X_{11}-X_{22}), \\
X_{32}(Y_{11}-Y_{22}) = Y_{32}(X_{11}-X_{22}), \\ 
 X_{21}Y_{32} =X_{32}Y_{21}.
\end{cases}
\]
Denote $X_0 = X_{11}-X_{22}$ and $Y_0 = Y_{11} - Y_{22}$. Then $\mathcal{Y}_+'$
is
\[
\Spec \CC[X_{11},X_{21},X_{31},X_{32},X_0,X_{32}',Y_{11},Y_{21},Y_{31},Y_{32},Y_0,Y_{32}',g_{ij}]
\]
cut out by the equations
\begin{align*}
\begin{cases}
g_{22}X_{32}' = g_{33}X_{32}, \\
g_{22}X_0 = g_{23}X_{32}, \\
g_{33}X_0 = g_{23}X_{32}'.
\end{cases},
\begin{cases}
g_{22}Y_{32}' = g_{33}Y_{32}, \\
g_{22}Y_0 = g_{23}Y_{32}, \\
g_{33}Y_0 = g_{23}Y_{32}'.
\end{cases}
\begin{cases}
X_{21}Y_0 = Y_{21}X_0, \\
X_{32}Y_0 = Y_{32}X_0,\\
X_{21}Y_{32} = X_{32}Y_{21}.
\end{cases}
\end{align*}

On the open subset such that $g_{23}\neq 0$, we have $X_{32} = \frac{g_{22}X_0}{g_{23}}$ and $X_{32}' = \frac{g_{33}X_0}{g_{23}}$ and similarly for $Y$'s. The remaining equation is
\[
X_{21}Y_0 = Y_{21}X_0,
\]
and hence reduced.

On the open subset that $g_{22}g_{33}\neq 0$, we have $X_{32}' = \frac{g_{33}X_{32}}{g_{22}}$ and $X_0 = \frac{g_{23}X_{32}}{g_{22}}$ and similarly for $Y$'s. Thus, the remaining equation is
\[
X_{21}Y_{32}=X_{32}Y_{21},
\]
and also reduced.
Since $\det g\neq 0$, we have either $g_{23}\neq 0$ or $g_{22}g_{33}\neq 0$,
thus $\mathcal{Y}'_{+}$ and $\mathcal{Y}$ are reduced. The similar arguments hold for $\mathcal{Y}_{-}$.
\end{proof}

\subsection{A Vanishing Argument}

\begin{definition}
  We define the following morphisms:
  \begin{equation}\label{diag:6}
    \begin{tikzcd}
      \mathcal{Y} \ar{d}{\pi^{\uparrow}} & \mathcal{Y}_{+} \ar{d}{\pi_+^{\uparrow}} &\mathcal{Y}_{-} \ar{d}{\pi_-^{\uparrow}} \\
      \widetilde{Flag_{x,y}^{\circ}} & \widetilde{Flag_{x,y,x}^{\circ}} & \widetilde{Flag_{x,x,y}^{\circ}}
    \end{tikzcd}
  \end{equation}
  obtained by remembering only the top part of the square in
  (\ref{diag:3})-(\ref{diag:5}).
  We also define the following morphisms:
  \begin{equation}\label{diag:7}
    \begin{tikzcd}
      \mathcal{Y} \ar{d}{\pi^{\downarrow}} & \mathcal{Y}_{+} \ar{d}{\pi_+^{\downarrow}} &\mathcal{Y}_{-} \ar{d}{\pi_-^{\downarrow}} \\
      \widetilde{Flag_{y,x}^{\circ}} & \widetilde{Flag_{y,x,x}^{\circ}} & \widetilde{Flag_{x,y,x}^{\circ}}
    \end{tikzcd}
  \end{equation}
   obtained by remembering only the bottom part of the square in
  (\ref{diag:3})-(\ref{diag:5}).
\end{definition}

\begin{prop}\label{prop:Y-embeds-into-P1-bundle} We have the following commutative diagram:
\[
\begin{tikzcd}
\mathcal{Y}\arrow[hook]{r}{\iota^{\uparrow} \text{or }
  \iota^{\downarrow}}\arrow{dr}{\pi^{\uparrow} \text{or} \pi^{\downarrow}} & \PP_{\widetilde{Flag_{1,1}^\circ}} (\mathcal{U}_2) \arrow{d} \\
& \widetilde{Flag_{1,1}^\circ} 
\end{tikzcd}
\]
where $\iota^\uparrow$ and $\iota^{\downarrow}$ are closed embeddings. The same
property holds for the spaces $\mathcal{Y}_+,\mathcal{Y}_{-}$ while replacing
$\widetilde{Flag_{x,y}^{\circ}}$ by the corresponding moduli spaces in
(\ref{diag:6}) and (\ref{diag:7}) and replacing $\mathcal{U}_2$ by
$\mathcal{U}_2$ or $\mathcal{U}_3/\mathcal{U}_1$ respectively.
\end{prop}

\begin{proof}
  We will only prove the statement above for $\mathcal{Y}$, as the cases of
  $\mathcal{Y}_-$ and $\mathcal{Y}_+$ hold analogously.
  The inclusion of locally free sheaves $\LL_1'\subset \mathcal{U}_2$ induces the
  morphism $\iota:\mathcal{Y}\to \PP_{\widetilde{Flag_{1,1}^\circ}}$.  We claim that corresponding map $\iota$ is a closed embedding.

By Theorem 1.7.8 of \cite{L04}, it suffices to show that for each closed point
$p\in \widetilde{Flag_{1,1}^\circ}$, the map between the fibers $\iota_p:
\mathcal{Y}_p\to \PP_{\widetilde{Flag_{1,1}^\circ}} (\mathcal{U}_2)_p$ is a
closed embedding. When $\EE_2$ are supported in two different points or
supported in a single point but not semi-simple, $\mathcal{Y}_{p}$ is a closed
point. When $\EE_2$ is semi-simple and supported in a single point, i.e.
$\mathcal{E}_2\cong k_x^{\oplus 2}$ for $x\in S$, then $\mathcal{Y}_{p}=\PP^1$
and $\iota_p$ is an isomorphism.
\end{proof}

\begin{prop}\label{prop:pushforward-of-OY}
The morphism $\pi^{\uparrow}: \mathcal{Y}\to \widetilde{Flag_{1,1}^\circ}$ is proper and satisfies
\[
R^i\pi_*^\uparrow \OO_{\mathcal{Y}} =  \begin{cases}
\OO_{\widetilde{Flag_{1,1}^\circ}}, & \text{ if }i=0,\\
0, & \text{ if }i>0.
\end{cases}
\]
The analogous properties hold for $\pi^{\downarrow}$.  Moreover, the analogous
properties hold with the scheme $\mathcal{Y}$ replaced by the schemes
$\mathcal{Y}_+$ and $\mathcal{Y}_-$,
\end{prop}
\begin{proof} 
By Proposition \ref{prop:Y-embeds-into-P1-bundle}, we can embed $\mathcal{Y}$ into a $\PP^1$-bundle $\PP(\mathcal{U}_2)$ over $\widetilde{Flag_{1,1}^\circ}$. Denote the projection $\pi: \PP(\mathcal{U}_2)\twoheadrightarrow \widetilde{Flag_{1,1}^\circ}$, then
\[
R^i\pi_*(\OO_{\PP(\mathcal{U}_2)}) = 0 \text{ for all } i\geq 1,
\]
and for any coherent sheaf $\FF$ on $\PP(\mathcal{U}_2)$,
\begin{align*}
R^i\pi_*(\FF) = 0 \text{ for all } i\geq 2.
\end{align*}
Now, from the exact sequence
\[
0 \to \mathcal{K} \to \OO_{\PP_2(\mathcal{U}_2)} \to \iota_*\OO_{\mathcal{Y}} \to 0,
\]
(where $\mathcal{K}$ is the kernel sheaf) we obtain the long exact sequence
\[
\cdots \to R^i\pi_*(\OO_{\PP_2(\mathcal{U}_2)}) \to R^i\pi^{\uparrow}_*(\OO_{\mathcal{Y}}) \to R^{i+1}\pi_* (\mathcal{K}) \to \cdots
\]
This implies that $R^i\pi^{\uparrow}_*(\OO_\mathcal{Y}) = 0$ for $i\geq 1$. The $i=0$ case follows from Stein factorization and the following facts:
\begin{itemize}
    \item $\widetilde{Flag_{1,1}^\circ}$ is normal. (Lemma \ref{lemma:Flag_schemes_normal})

    \item $\mathcal{Y}$ is reduced. (Lemma \ref{lemma:Y,Y+-_reduced})

    \item $\pi^{\uparrow}$ is proper and all its fibers are either a point or $\PP^1$. 
\end{itemize}
\end{proof}

\begin{prop}\label{prop:zero-locus-Y}
On $\mathcal{Y}$ or $\mathcal{Y}_+$, the natural map 
\[
\LL_1=\mathcal{U}_{1}\subset \mathcal{U}_2 \to \mathcal{U}_{2}/\mathcal{U}_{1}'=\LL_2'
\]
induces a global section of $\LL_{2}'\otimes \LL_1^{-1}$, with the zero section consisting of the data
\[
\{(\EE_1,x) = (\EE_1',y)\} \subset \mathcal{Y},
\]
which is isomorphic to $\widetilde{Flag_{x,x}^\circ}$ or $\widetilde{Flag_{x,x,x}^\circ}$.

Analogously, on $\mathcal{Y}_{-}$, there is a global section of
$\LL_{3}'\otimes \LL_2^{-1}$ such that the zero section is isomorphic to $\widetilde{Flag_{x,x,x}^\circ}$.

\end{prop}

\begin{proof} We will only prove the case on $\mathcal{Y}$. The morphism is identity when $\mathcal{E}_2$ is supported in two
  different points and always zero when $\mathcal{E}_2$ is supported in one
  point and not semisimple. When $\mathcal{E}_2$ is supported in one
  point and semisimple, then the morphism vanishes if and only if $\EE_1=\EE_1'$.
\end{proof}

\section{$K$-theoretic Hall algebra on Surfaces}
In this section, we review the $K$-theoretic Hall algebra on surfaces in
\cite{zhao2019k}. Given a reductive group $G$ acting on $X$, we denote $K_G(X)$
the $G$-equivariant Grothendieck groups of coherent sheaves on $X$.

\subsection{Refined Gysin maps and Derived Fiber Squares}
A morphism $f:X\to Y$ is called a \textit{local complete intersection} (l.c.i.) morphism if $f$ is the composition of a regular embedding and a smooth morphism. In this case, $f$ has finite Tor dimension.

\begin{definition}
Suppose we have a Cartesian diagram
\[
\begin{tikzcd}
 X' \arrow{r}\arrow{d}{f'} & X \arrow{d}{f} \\
 Y' \arrow{r} & Y 
\end{tikzcd}
\]
where $f$ is a l.c.i. morphism. The refined Gysin map $f^!: K(Y')\to K(X')$ is defined by
\[
f^!([\FF]) = \sum_{i}(-1)^i [\text{Tor}_i^{\OO_Y}(\OO_X,\FF)].
\]
\end{definition}

Consider a Cartesian diagram
\begin{equation}\label{fig:derived-square}
    \begin{tikzcd}
X' \arrow[hook]{r}{g} \arrow{d}{f'} & X \arrow{d}{f} \\
Y' \arrow[hook]{r}{g'} & Y
\end{tikzcd}
\end{equation}
where $f: X\to Y$ is a l.c.i. morphism and $g': Y'\to Y$ is a closed embedding.

\begin{definition}
In diagram (\ref{fig:derived-square}), if $X'$ has the expected dimension, i.e., if
\[
\dim X' - \dim Y' = \dim X - \dim Y,
\]
then we say that (\ref{fig:derived-square}) is a \textit{derived fiber square}.
\end{definition}

A derived fiber square has the following property:
\begin{prop}\label{prop:refined_Gysin_map_agrees_with_usual_pullback}
Suppose (\ref{fig:derived-square}) is a derived fiber square, and $Y'$ is
Cohen-Macaulay. Then $X$ is Cohen-Macaulay and $f'^*([\mathcal{F}])=f^!([\mathcal{F}])$ for any locally free
sheaf $\mathcal{F}$ on $Y'$.
\end{prop}
\begin{proof}
  Without the loss of generality, we assume that $f$ is a regular embedding. The
  problem is local, and thus we assume that $Y=Spec A$ for a local ring $A$, and
  $X=Spec A/(f_1,\cdots,f_k)$ for some $f_1,\cdots f_k$ in $A$ such that
  $k=dim Y-dim X$. Moreover, by Lemma 2.8 of \cite{zhao2019k}, we can assume
  that $k=1$. Let $Y'=Spec B$ for another Cohen-Macaulay local ring. By Krull's
  principal ideal theorem, $f_1$ is a non-zero divisor in $B$ and thus
  $Tor^i_A(\mathcal{F},A/(f_1))=0$ if $i>0$ when $\mathcal{F}$ is locally free.
\end{proof}

\subsection{$K$-theoretic Hall Algebra on Surfaces}
The $K$-theoretic Hall algebra on $S$ \cite{zhao2019k} is an associative algebra
structure on
\[
K(Coh)=\bigoplus_{n=0}^\infty K_{\GL_n}(Quot_n^\circ).
\]
The algebra structure is given as follows:

\begin{itemize}
\item For any two non-negative integers $n,m$, there is a linear morphism of
  locally free sheaves $\psi_{n,m}:W_{n,m}\to V_{n,m}$ on $Quot_{n,m}^{\circ}$
  with the Cartesian diagram:
  \[
\begin{tikzcd}
Flag_{n,m}^\circ \arrow{r}\arrow{d} & W_{n,m}\arrow{d}{\psi_{n,m}}\\
Quot_n^\circ\times Quot_m^\circ \arrow{r} & V_{n,m}
\end{tikzcd}
\]
which induces the refined Gysin map:
\[\psi_{n,m}:K_{GL_n}(Quot_n)^{\circ}\times K_{GL_m}(Quot_m^{\circ})\to
  K_{P_{n,m}}(Flag_{n,m}^{\circ}).\]
Moreover, $$dim W_{n,m}-dim
V_{n,m}=hom(\mathcal{I}^n,\mathcal{E}_m)-ext^1(\mathcal{I}^n,\mathcal{E}_m)=nm$$
by the Grothendieck-Riemann-Roch formula.
\item Formula (5.2.17) of \cite{chriss2009representation} induces the
  isomorphism:
\[ind_{P_{n,m}}^{GL_{n+m}}:  K_{P_{n,m}}(Flag_{n,m}^{\circ})\cong
  K_{GL_{n+m}}(\widetilde{Flag_{n,m}^{\circ}})\]
\item The morphism $q_{n,m}:\widetilde{Flag_{n,m}^{\circ}}\to
  Quot_{n+m}^{\circ}$ is
  proper and induces a push forward morphism:
  $q_{n,m*}:K_{GL_{n+m}}(\widetilde{Flag_{n,m}^{\circ}})\to
  K_{GL_{n+m}}(Quot_{n+m}^{\circ})$.
\item We define the algebraic structure $*:K_{GL_n}(Quot_n^{\circ})\times
  K_{GL_m}(Quot_m^{\circ})\to K_{GL_{n+m}}(Quot_{n+m}^{\circ})$ as the
  composition of $\psi_{n,m}^{!},ind_{P_{n,m}}^{GL_{n+m}}$ and $q_{n,m*}$.
\end{itemize}
\begin{definition}
  We define the $K$-theoretic class $e_{(d_1,\ldots,d_n)}\in
  K_{GL_n}(Quot_n^{\circ})$ when $n\leq 3$ by  $e_k=[z^k\mathcal{O}_S]$ where
  $z$ is the standard character  of $\mathbb{G}_{m}$ and 
  \begin{align*}
       e_{k_1,k_2}=q_{1,1*}(\mathcal{L}_1^{k_1}\mathcal{L}_{2}^{k_2}\mathcal{O}_{\widetilde{Flag_{x,x}^\circ}}); \\
    e_{k_1,k_2,k_3}=q_{1,1,1*}(\mathcal{L}_1^{k_1}\mathcal{L}_{2}^{k_2}\mathcal{L}_3^{k_3}\mathcal{O}_{\widetilde{Flag_{x,x,x}^\circ}})
  \end{align*}
 \end{definition}

\begin{prop}\label{prop:commutator_relation_n=2}
  Given two integers $d\geq k$,
  \[
[e_{d},e_{k}] := e_{d} * e_{k} - e_{k}*e_{d} = \sum_{a=k}^{d-1} e_{a,d+k-a} \in K_{\GL_2}(Quot_2^\circ).
\]
\end{prop}
\begin{proof}
  By Lemma \ref{lemma:Flag_schemes_Cohen_Macaulay}, $Flag_{x,y}^{\circ}$ is
  Cohen-Macaulay of expected dimensions. By Proposition \ref{prop:refined_Gysin_map_agrees_with_usual_pullback}, we have
  \begin{align*}
    e_d*e_k=q_{1,1*}(\mathcal{L}_1^d\mathcal{L}_2^k)=q_{1,1*}\pi^\uparrow_*(\mathcal{L}_1^d\mathcal{L}_2^k) \\
    e_ke_d=q_{1,1*}(\mathcal{L}_1^k\mathcal{L}_2^l)=q_{1,1*}\pi^{\downarrow}_*(\mathcal{L}_1'^k\mathcal{L}_2'^{d})= q_{1,1*}\pi^{\uparrow}_*(\mathcal{L}_1'^k\mathcal{L}_2'^{d})\\
    e_{i,j}=q_{1,1*}\pi^{\uparrow}_*(\mathcal{L}_1^i\mathcal{L}_2^{j}\mathcal{O}_{\widetilde{Flag_{x,x}^\circ}})
  \end{align*}
  
  Hence we only need to prove that
  $$[\LL_1^d\LL_2^k]-[\LL_1'^k\LL_2'^d]= \sum_{i=0}^{d-k-1}
  [\LL_1^{k+i}\LL_2'^{d-i}\mathcal{O}_{\widetilde{Flag_{x,x}^\circ}}]$$ on $K_{GL_2}(\mathcal{Y})$.

By Proposition \ref{prop:zero-locus-Y}, there is an exact sequence of $\GL_2$-equivariant coherent sheaves on $\mathcal{Y}$:
\[
0\to \LL_1'^{-1}\otimes \LL_2 \to \OO_\mathcal{Y} \to \OO_{\widetilde{Flag_{x,x}^\circ}} \to 0
\]
Since $\LL_1\LL_2 \cong \LL_1'\LL_2'$, on $K_{\GL_2}(\mathcal{Y})$ we have
\[
[\LL_1]-[\LL_2'] = [\OO_{\widetilde{Flag_{x,x}^\circ}}\otimes \LL_2'] = [\OO_{\widetilde{Flag_{x,x}^\circ}}][\LL_2']
\]
From this, we can calculate on $K_{\GL_2}(\mathcal{Y})$ that
\begin{align*}
 [\LL_1^d\LL_2^k]-[\LL_1'^k\LL_2'^d] &= [\LL_1^k\LL_2^k]([\LL_1]^{d-k}-[\LL_2']^{d-k}) \\
 &= [\LL_1^k\LL_2^k]([\LL_1]-[\LL_2'])\sum_{i=0}^{d-k-1} [\LL_1^i\LL_2'^{d-k-i-1}] \\
                                     &= [\LL_1^k\LL_2^k][\LL_2'][\OO_{\widetilde{Flag_{x,x}^\circ}} ]\sum_{i=0}^{d-k-1} [\LL_1^i\LL_2'^{d-k-i-1}] \\
   &= \sum_{i=0}^{d-k-1} [\LL_1^{k+i}\LL_2^{d-i}\OO_{\widetilde{Flag_{x,x}^\circ}} ]
\end{align*}
as on $\widetilde{Flag_{x,x}^\circ}$, $\LL_1' \cong \LL_1$ and $\LL_2' \cong \LL_2$.
\end{proof}

\begin{prop}\label{prop:commutator_relation_n=3}
Suppose $d_1,d_2,k$ are integers, then
\[
[e_{d_1,d_2},e_{k}] =\begin{cases}
-\sum_{a=d_1}^{k-1}e_{a,d_1+k-a,d_2}, &\text{ if }k \geq d_1,\\
\sum_{a=k}^{d_1-1}e_{a,d_1+k-a,d_2}, &\text{ if }k < d_1.
\end{cases} + \begin{cases}
-\sum_{a=d_2}^{k-1}e_{d_1, a, d_2+k-a}, &\text{ if }k \geq d_2,\\
\sum_{a=k}^{d_2-1}e_{d_1, a, d_2+k-a}, &\text{ if }k < d_2.
\end{cases}
\]
\end{prop}
\begin{proof}
As
$Flag_{x,x}^{\circ},Flag_{x,y}^{\circ},Flag_{x,y,z}^{\circ},Flag_{x,y,x}^{\circ},Flag_{x,x,y}^{\circ},Flag_{y,x,x}^{\circ}$
are all Cohen-Macaulay of expected dimension, we have
\begin{align*}
  e_{d_1,d_2}*e_k=p_{1,1,1*}([\mathcal{L}_{1}^{d_1}\mathcal{L}_2^{d_2}\mathcal{L}_{3}^k\mathcal{O}_{\widetilde{Flag_{x,x,y}^{\circ}}}]) \\
  e_{k}*e_{d_1,d_2}=p_{1,1,1*}([\mathcal{L}_{1}^{k}\mathcal{L}_2^{d_1}\mathcal{L}_{3}^{d_2}\mathcal{O}_{\widetilde{Flag_{y,x,x}^{\circ}}}])
\end{align*}
We define
$f_{d_1,k,d_2}=p_{1,1,1*}([\mathcal{L}_{1}^{d_1}\mathcal{L}_2^{k}\mathcal{L}_{3}^{d_2}\mathcal{O}_{\widetilde{Flag_{x,y,x}^{\circ}}}])$,
then by the similar argument as Proposition \ref{prop:commutator_relation_n=2}, we have 
\begin{align*}
  e_{d_1,d_2}*e_{k}-f_{d_1,k,d_2} =\begin{cases}
-\sum_{a=d_1}^{k-1}e_{a,d_1+k-a,d_2}, &\text{ if }k \geq d_1,\\
\sum_{a=k}^{d_1-1}e_{a,d_1+k-a,d_2}, &\text{ if }k < d_1.
\end{cases}
\end{align*}
\begin{align*}
  f_{d_1,k,d_2}-e_k*e_{d_1,d_2}= \begin{cases}
-\sum_{a=d_2}^{k-1}e_{d_1, a, d_2+k-a}, &\text{ if }k \geq d_2,\\
\sum_{a=k}^{d_2-1}e_{d_1, a, d_2+k-a}, &\text{ if }k < d_2.
\end{cases}
\end{align*}
\end{proof}

\subsection{The proof of Theorem \ref{main}}

Theorem \ref{main} directly follows from Proposition
\ref{prop:commutator_relation_n=2} and Proposition
\ref{prop:commutator_relation_n=3} as

\begin{align*}
  [[e_{k+1},e_{k-1}],e_k]&= [e_{k-1,k+1}+e_{k,k},e_k] \\
                         &= e_{k-1,k,k+1}-e_{k-1,k,k+1}\\
  &=0.
\end{align*}

\bibliographystyle{plain}
\bibliography{write-up}
\end{document}